 \documentclass[11pt]{article}
 \UseRawInputEncoding
 \usepackage{bbm}
\usepackage{amssymb, amsthm, amsmath, amscd}
\usepackage{hyperref}

\newtheorem{thm}{Theorem}[section]
\newtheorem{lem}[thm]{Lemma}
\newtheorem{prop}[thm]{Proposition}
\newtheorem{cor}[thm]{Corollary}
\newtheorem{NN}[thm]{}

\newtheorem{QE}[thm]{Question}
\theoremstyle{definition}\newtheorem{df}[thm]{Definition}
\theoremstyle{definition}\newtheorem{rem}[thm]{Remark}
\theoremstyle{definition}\newtheorem{exm}[thm]{Example}

\renewcommand{\phi}{\varphi}

\newcommand{\N}{\mathbb{N}}
\newcommand{\Z}{\mathbb{Z}}

\newcommand{\R}{\mathbb{R}}
\newcommand{\C}{\mathbb{C}}
\newcommand{\T}{\mathbb{T}}
\newcommand{\I}{{\mathbb{I}}}

\newcommand{\hm}{homomorphism}
\newcommand{\dt}{\delta}
\newcommand{\ep}{\epsilon}
\newcommand{\la}{\langle}
\newcommand{\ra}{\rangle}
\newcommand{\andeqn}{\,\,\,{\rm and}\,\,\,}
\newcommand{\rforal}{\,\,\,{\rm for\,\,\,all}\,\,\,}
\newcommand{\CA}{$C^*$-algebra}
\newcommand{\SCA}{$C^*$-subalgebra}

\newcommand{\af}{{\alpha}}
\newcommand{\bt}{{\beta}}
\newcommand{\dist}{{\rm dist}}

\newcommand{\D}{\mathbb D}
\newcommand{\beq}{\begin{eqnarray}}
\newcommand{\eneq}{\end{eqnarray}}
\newcommand{\tforal}{\,\,\,\text{for\,\,\,all}\,\,\,}
\newcommand{\tand}{\,\,\,\text{and}\,\,\,}

\newcommand{\Wlog}{Without loss of generality}
\newcommand{\diag}{{\rm diag}}

\newcommand{\td}{\tilde}

\title{Existence  of Approximately Macroscopically Unique  States}
\author{Huaxin Lin}

\date{}

\begin{document}

\maketitle

\begin{abstract}
Let $H$ be an infinite dimensional separable Hilbert space and 
$B(H)$ the \CA\, of bounded operators on $H.$ Suppose that
$T_1,T_2,..., T_n$ are self-adjoint operators in $B(H).$
We show that, if commutators  $[T_i, T_j]$ are sufficiently small in norm, then   ``Approximately Macroscopically 
Unique" states always exist for any values in a synthetic spectrum 
of the $n$-tuple of self-adjoint operators. 
This is achieved under the circumstance  for which 
the $n$-tuple  may not be approximated by commuting ones. This answers a question proposed by
David Mumford for measurements in quantum theory.  If commutators are not small
in norm but small modulo compact operators, then ``Approximate Macroscopic Uniqueness"
states also exist.  
\end{abstract}

\section{Introduction}
In quantum mechanics, macroscopic observables may be represented by bounded 
self-adjoint operators $T_1, T_2,...,T_n$  in a Hilbert space $H.$ 
Commutators $[T_iT_j,T_jT_i]$ are related  to the uncertainty principle in their measurements and 
small commutators indicate more precise measurements.  
Let $v\in H$ with $\|v\|=1$ which gives a vector state.  The expected value  when measuring  an observable given by  $T_j$
 is 
${\rm exp}_{T_j}(v)=\la T_jv, v\ra,$ $1\le j\le n,$ (see p.178 of \cite{Mf2}).
Most interesting vector states are eigenstates, the vector states given by eigenvectors. 
The joint expected value associated with $v$ is $({\rm exp}_{T_1}(v), {\rm exp}_{T_2}(v),...,{\rm exp}_{T_n}(v))\in \R^n.$
However, 
it is difficult to get states which are joint eigenstates.  Strictly speaking, one should not even expect 
to have any eigenvector states at all much less the joint eigenvectors.
 In fact, a self-adjoint operator may not have any eigenvalues,
in general. But every point in the spectrum of a self-adjoint operator is an approximate eigenvalue.

In his recent book \cite{Mf2}, David Mumford proposed to study ``near eigenvectors"  for some set of human observables
which are called ``Approximately Macroscopically Unique" states.  Mumford explained that these 
states describe a world recognizable to us with no maybe-dead/maybe-alive cats. 
Given a vector state $v\in H,$  variance and standard deviation of the measurement made 
by $T_j$ are defined by ($1\le j\le n$)
\beq
{\rm var}_{T_j}(v)=\la (T_j-\exp_{T_j}(v)I)^2 v, v\ra\andeqn\\
{\rm sd}_{T_j}(v)=\sqrt{{\rm var}_{T_j}(v)}=\|(T_j-{\rm exp}_{T_j}I)v\|.
\eneq
 
 The following definition was given by  Mumford (p. 179,  II. AMU states, Chapter 14 of \cite{Mf2}):
 \begin{df}\label{DAMU}
 Let $T_1, T_2,..., T_n$ be a given $n$-tuple of self-adjoint bounded operators.  Define
 \beq\label{DAMU-1}
 {\rm AMU}(\{T_j: 1\le j\le n; \sigma\})=\{ v\in H: \|v\|=1, \,\,\, {\rm sd}_{T_j}(v)<\sigma, \,\, 1\le j\le n\},
 \eneq
 where $\sigma>0$ is a given tolerance. 
 \end{df}
 Following Mumford, immediately one has the question:
 \begin{QE}\label{Qamu-1}
 Is the set in \eqref{DAMU-1} non-empty when the commutators $[T_iT_j-T_jT_i]$ are sufficiently small?
 \end{QE}

 As pointed out by Mumford that it is natural to assume that the commutators have small norm as 
 operators in the \CA\, $B(H)$ of all bounded linear operators on $H,$ a constraint on macroscopic 
 variables (see p. 118, Chapter 14 of \cite{Mf2}). 
   In the First International Congress of Basic Science held in 2023,  David Mumford
gave the opening  plenary lecture on Consciousness, robots and DNA (\cite{Mf}).  At the very end of this
magnificent  lecture,
he asked whether there are, nearby, within a small tolerance, an $n$-tuple of commuting self-adjoint 
operators so one may project them to a common eigenvector subspace.  
A version of the  question may be reformulated as follows:

\begin{QE}\label{QQ}
Let $\ep>0$ and $n\in \N$ be a positive integer.
When is  there a constant $\dt>0$ such that  the following statement holds?
 If  $H$ is  any  separable  Hilbert space and 
 $T_1, T_2,...,T_n\in B(H)$ are self-adjoint  with $\|T_j\|\le 1$ ($1\le j\le n$) such that
 \beq\label{mfp-1}
 \|T_jT_i-T_iT_j\|<\dt,\,\,\, 1\le i,j\le n,
 \eneq
there exist self-adjoint operators $S_1,S_2,...,S_n$ on $H$
such that
\beq\label{mfp-2}
S_jS_i=S_iS_j\andeqn \|S_i-T_i\|<\ep,\,\, \, 1\le i,j\le n.
\eneq

\end{QE}
 
When  $n=2$ and $H$  is any finite dimensional Hilbert space  (no bound on the dimension)
 the  same problem is known 
as von-Neumann-Kadison-Halmos problem for almost commuting self-adjoint matrices
(see \cite{Halmos1} and \cite{Halmos2}).  In that case, the answer is that such $\dt$ always exists
(independent of the dimension)
\cite{Linmatrices} (see also \cite{FR}  and  \cite{Has}).

In quantum mechanics, one often heard `` when commutators tend  to zero, or 
the quantity $\hbar\to 0,$ we recover the classical system". 
Indeed, if observables in the system are compact self-adjoint operators, by
the affirmative solution to von-Neumann-Kadison-Halmos problem  above (\cite{Linmatrices}), 
a pair of  observables can be approximated by commuting observables when  the norm of commutators
$[T_1, T_2]$  is small ($n=2$).  In general, when these observables are not compact,
as pointed out  by Mumford (see  also Example 4.6 of \cite{FR}) there is a  serious problem with this statement.
If in  classical system, observables are commuting, and in quantum systems, observables have non-zero (but 
small) commutators, then ``$\hbar\to 0,$ or commutators tend to zero" do not necessarily  recover 
the classical system. In fact, in general, no matter how small $\hbar$ or commutators are,
observables could be  far away from any commuting ones. 
There are topological obstacles  (the obvious one is Fredholm index and  there are hidden ones---see 
Proposition 5.5 and Proposition 6.4 in \cite{Linself}). 
It is shown in \cite{Linself} that when an approximate synthetic spectrum and essential 
synthetic spectrum of the $n$-tuple are close, the answer  to \ref{QQ} is affirmative. 
This result may be  interpreted as 
when one assumes that any ``local" measurement is not too far off 
from some ``outside" measurements, when commutators are sufficiently small, observables 
can be approximate by commuting ones (see Remark 6.6 of \cite{Linself}).

But, of course,  Mumford is right about the existence of AMU states.
Instead of trying to find a condition  on a
quantum system with small commutators which allows one to find near-by commuting observables, 
we may directly  answer Question \ref{Qamu-1}.
Theorem \ref{IT-1} below states that, 
when commutators are sufficiently small, there are indeed always some 
approximately macroscopically unique (pure) states  for multiple observables simultaneously.

\begin{thm}\label{IT-1}
Let $n\in \N$ and $\ep>0.$  There exists $\dt(n, \ep)>0$ satisfying the following:
Suppose that $H$ is an infinite dimensional separable Hilbert space and $T_1, T_2,...,T_n\in B(H)$
are self-adjoint operators with $\|T_j\|\le 1$ ($1\le j\le n$) such that
\beq\label{T1-c}
\|T_iT_j-T_jT_i\|<\dt,\,\,\,\, 1\le i,j\le n.
\eneq
Then, for any $\lambda=(\lambda_1, \lambda_2,...,\lambda_n)\in s{\rm Sp}^{\ep/4}((T_1,T_2,...,T_n)),$
there exists $v\in H$ with $\|v\|=1$ such that
\beq\label{T1-c-1}
&&\max_{1\le i\le n}|{\rm exp}_{T_i}(v)-\lambda_i|=\max_{1\le i\le n}|\la T_iv,v\ra-\lambda_i|<\ep\andeqn\\\label{T1-c-2}
&&\max_{1\le i\le n}\{{\rm sd}_{T_j}(v)\}=\max_{1\le i\le n}\|(T_i-{\rm exp}_{T_i}(v)\cdot I)v\|<\ep.
\eneq
\end{thm}
Here   $s{\rm Sp}^\eta((T_1, T_2,..., T_n))$ is the set of $\eta$-synthetic-spectrum 
(see Definition  \ref{Dpseudosp}). 
This shows that in a macroscopic system, if one assumes a suitable commutator bound,
AMU states always exist.  Moreover,  
every point $\lambda=(\lambda_1, \lambda_2,...,\lambda_n)$ in $s{\rm Sp}^{\ep/4}((T_1,T_2,...,T_n))$
can be  approximated by a joint expected value for which AMU states exist. 

The universe is even more interesting.  We could have large commutators in norm but ``tend to zero" in a different sense.
Let us assume that $H$ is an infinite dimensional separable Hilbert space
with an orthonormal basis 
$\{u_k: k\in \N\}$ of $H.$ Let $p_m$ be the projection 
on ${\rm span}\{u_k: 1\le k\le m\}.$ 
Suppose that
\beq\label{compct-1}
\lim_{m\to\infty}\|(T_iT_j-T_jT_i)(1-p_m)\|=0,\,\,\, 1\le i,j\le n.
\eneq
In other words, the commutators are compact.

%
Denote by ${\cal K}$ the \SCA\, of all compact operators. 
Then we have the following:
\begin{thm}\label{IT2}
Let $H$ be an infinite dimensional  separable Hilbert space and 
$T_1, T_2,...,T_n\in B(H)$ be self-adjoint operators.
Suppose that
\beq\label{IT2-0}
T_iT_j-T_jT_i\in {\cal K},\,\,\, 1\le i,j\le  n.
\eneq
Then, for any $\lambda=(\lambda_1, \lambda_2,...,\lambda_n)\in {\rm Sp}(\pi_c(T_1), \pi_c(T_2), ...,\pi_c(T_n)),$
where $\pi_c: B(H)\to B(H)/{\cal K}$ is  the quotient map,
there exists  a sequence $v^{(m)}\in H$ with $\|v^{(m)}\|=1$ such that
\beq\label{IT2-1}
&&\lim_{m\to\infty}\max_{1\le i\le n}|\lambda-{\rm exp}_{Ti}(v^{(m)})|=0\tand\\\label{IT2-2}
&&\lim_{m\to\infty}\max_{1\le i\le n}{\rm sd}_{T_i}(v^{(m)})=\lim_{m\to\infty}\max_{1\le i\le n}\|(T_j-{\rm exp}_{T_i}\cdot I)(v^{(m)})\|=0.
\eneq
Moreover, for any $\zeta=(\zeta_1, \zeta_2,...,\zeta_n)\subset {\rm Conv}({\rm Sp}(T_1, T_2,...,T_n)),$
there are unit vectors $u^{(k)}\in H$ such that
\beq\label{IT2-3}
\lim_{k\to\infty}\max_{1\le j\le k}|\zeta_j-\la T_j u^{(k)}, u^{(k)}\ra\|=0.
\eneq
\end{thm}

Here ${\rm Sp}(\pi_c(T_1), \pi_c(T), ...,\pi_C(T_n))$ is  the essential joint spectrum of the $n$-tuple 
self-adjoint operators  (see Definition \ref{DSP}).  Note that any $\eta$-essential synthetic spectrum 
contains (see Definition \ref{Dpseudosp}  and Proposition \ref{Pspuniq}) ${\rm Sp}(\pi_c(T_1), \pi_c(T), ...,\pi_C(T_n)).$  
These are expected values 
of pure quantum states of $B(H),$ or $C,$ the \SCA\, generated by $T_1, T_2,...,T_n, 1$ vanishing on ${\cal K}.$
These may be called essential  joint expected values. 

The condition \eqref{IT2-0} means that the commutators tend to zero along any orthonormal basis 
of $H$ in the sense of \eqref{compct-1}.  It may be interpreted that any measurement with less local 
(finite dimensional) interference 
increases its accuracy.   Another way to interpret  this  is  the commutators (or $\hbar$) vanishes at the far edge of 
the (non-commutative) universe, or, in operator algebra's term, it vanishes in the corona algebra. 
 With this condition on the macroscopic observables, 
Theorem \ref{IT2} states that, any essential joint expected value can  be measured 
(simultaneously) by 
an AMU state.  It  should be noted  (see (4) of Remark \ref{catdiscussion})  that the condition \eqref{IT2-0} 
does {\it not} imply that the system is a compact perturbation of classical ones.

In section 2, apart from some notations, we provide some review on approximate synthetic spectrum 
for $n$-tuples of self-adjoint operators. Section 3 presents the proof of Theorem \ref{IT-1} and 
some preparation for section 4. In section 4, we first prove Theorem \ref{TC+ep} which goes beyond 
both Theorem \ref{IT-1} and Theorem \ref{IT2}. We also include  a final remark. 

\section{Synthetic spectra}

\begin{df}\label{D1}
Let $A$ be a \CA. Denote by $A_{s.a.}$ the set of self-adjoint elements in $A.$

If $x, y\in A$ and $\ep>0,$ we write 
\beq
x\approx_\ep y,\,\,\, {\rm if}\,\,\, \|x-y\|<\ep.
\eneq

\end{df}

\begin{df}\label{DHd}
Let $\Omega$ be a metric space.
Denote by $C(\Omega)$ the \CA\,  of  all (complex valued) continuous functions on $\Omega.$

Suppose that $x\in \Omega$ and $r>0.$
Define  $B(x, r)=\{y\in \Omega: {\rm dist}(y,x)<r\}.$
For any subset $Y\subset \Omega,$ $\overline{Y}$ is the closure of $Y.$

If $X\subset \Omega$ is a compact subset and $\eta>0,$  denote by
$X_\eta=\{y\in \Omega: {\rm dist}(y,X)<\eta\}.$ 

Recall that the Hausdorff distance of two compact subsets of $X, Y\subset \Omega$ 
is defined by
\beq
d_H(X, Y)=\max\{\sup_{x\in X} \{{\rm dist}(x, Y)\}, \sup_{y\in Y}\{{\rm dist}(y, X)\}\}.
\eneq
Let $F(\Omega)$ be the set of all non-empty compact subsets of $M.$ 
Then $(F(\Omega), d_H)$ is a compact metric space with the metric $d_H.$

\end{df}

\begin{df}\label{DSP}
Let $e_0(\xi)=1$ for all $\xi\in \R^n$ be the constant function,
$e_i$ be a continuous function defined on $\R^n$ by
$e_i((r_1, r_2,...,r_n))=r_i$ for $(r_1, r_2,...,r_n)\in \R^n,$ $i=1,2,...,n.$
{\bf This notation will be used throughout this paper.}

Let $\I^n=\{(r_1, r_2, ...,r_n)\in \R^n: |r_i|\le 1\}.$ 
Note that $C(\I^n)$ is generated by $\{e_i|_{\I^n}: 0\le i\le 1\}.$


\end{df}

\begin{df}\label{Dsp1}
Let $A$ be a unital \CA\, and $x\in A.$ Denote by ${\rm sp}(x)$ the spectrum of $x.$
Suppose that $T_1,T_2,...,T_n\in A_{s.a.}$ and 
$T_iT_j=T_jT_i,$ $1\le i, j\le n.$ 
Let $C$ be the unital \SCA\, generated by $1, T_1,T_2,..., T_n.$
Put $T_0=1.$
Then there exists a compact subset $\Omega\subset \R^n$ 
and an isomorphism $\phi: C(\Omega)\to C$ such that
$\phi(e_j)=T_j,$ where $e_j$ is defined above.
In other words, $\{T_j: 0\le j\le n\}$ generates a \SCA\, $C\cong C(\Omega).$

The set  $\Omega$ is called the spectrum of the $n$-tuple self-adjoint operators
$T_1, T_2,...,T_n.$ We will write 
${\rm Sp}((T_1, T_2,...,T_n))=\Omega.$
\end{df}

\begin{df}\label{Dpic}
Let $H$ be an infinite dimensional separable Hilbert space.
For $x, y\in H,$ denote by $\la x, y\ra$ the inner product of $x$ and $y.$
Denote by $B(H)$ the \CA\, of all bounded linear operators on $H$ 
and by ${\cal K}$ the \CA\, of all compact operators on $H.$
Denote by $\pi_c: B(H)\to B(H)/{\cal K}$ the quotient map. 

Let $T\in B(H).$ Recall that the essential spectrum of $T,$ ${\rm sp}_{ess}(T)={\rm sp}(\pi_c(T)),$
is the spectrum of $\pi_c(T).$
\end{df}

\begin{df}\label{n-net}
Fix an integer $k\in \N$ and $M\ge 1.$   Define
%
$$P_k^M=\{\xi=(x_1,x_2,...,x_n): x_j=m_j/k,  |m_j|\le Mk, m_j\in \Z,\,\,1\le j\le n\}.$$
$P_k^M$ has only finitely many points.
\end{df}

\begin{df}\label{Dpseudosp}
Let $n\in \N$ and $M>0.$
In what follows,  for each $0<\eta<1,$ we  choose a fixed integer $k=k(\eta)\in \N$
such that $k=\inf\{l\in \N: (M+1)/l<{\eta\over{2\sqrt{n}}}\}.$ Denote $D^\eta=P_k^M.$
We write $D^\eta=\{x_1,x_2,...,x_m\}.$ Then $D^\eta$ is $\eta/2$-dense  in $\I^n.$
Moreover, $D^\eta\subset D^\dt$ if $0<\dt<\eta.$

Denote by $M^n=\{(r_1, r_2,...,r_n): |r_i|\le M\}.$

Fix $\xi=(\lambda_1, \lambda_2,...,\lambda_n)\in M^n.$
Let 
$\theta_{\lambda_i,\eta}\in C([-M, M])$ be such that 
$0\le \theta_{\lambda_i, \eta}\le 1,$ $\theta_{\lambda_i, \eta}(t)=1,$ if $|t-\lambda_i|\le 3\eta/4,$ 
$\theta_{\lambda_i,\eta}(t)=0$
if $|t-\lambda_i|\ge  \eta,$  and $\theta_{\lambda_i \eta}$ is linear in $(\lambda_i-\eta, \lambda_i-3\eta/4)$
and in $(\lambda_i+3\eta/4, \lambda_i+\eta),$ 
$i=1,2,...,n.$ 
 Define, for each $t=(t_1,t_2,...,t_n)\in \R^n,$ 
\beq\label{Dpsp-5}
\Theta_{\xi, \eta}(t)&=&\prod_{i=1}^n\theta_{\lambda_i, \eta}(t_i).
\eneq
Suppose that $A$ is a unital \CA\, and $(a_1,a_2,...,a_n)$ is an $n$-tuple  of self-adjoint elements in 
$A$ with $\|a_i\|\le M$ ($1\le i\le n$). 
Put
\beq
\Theta_{\xi, \eta}(a_1,a_2,...,a_n)&=&\theta_{\lambda_1,\eta}(a_1)\theta_{\lambda_2,\eta}(a_2)\cdots 
\theta_{\lambda_n,\eta}(a_n).
\eneq
Note that  we do not assume that $a_1, a_2,...,a_n$ mutually commute and  the product in \eqref{Dpsp-5} has 
a fixed order. 

For $x_j=(x_{j,1}, x_{j,2},...,x_{j,n})\in D^\eta,$
we may write $\theta_{j,i, \eta}:=\theta_{x_{j,i},\eta}$ and 
$\Theta_{j, \eta}:=\Theta_{x_j, \eta},$ $1\le i\le n,$ $j=1,2,...,m.$
Set 
\beq
s{\rm Sp}^\eta((a_1,a_2,..., a_n)=\bigcup_{_{\|\Theta_{j, \eta}(a_1,a_2,...,a_n)\|\ge 1-\eta}} \overline{B(x_j, \eta)}. 
\eneq
The set $s{\rm Sp}^\eta((a_1,a_2,...,a_n))$ is called $\eta$-synthetic-spectrum of the $n$-tuple 
$(a_1,a_2,...,a_n)$ which, by the definition, is compact. 

If $a_ia_j=a_ja_i$ for all $1\le i,j\le n,$ then 
$
{\rm Sp}((a_1, a_2,...,a_n))\subset s{\rm Sp}^\eta((a_1,a_2,...,a_n)).
$
\end{df}

The $\eta$-synthetic-spectrum can be tested. 
Suppose $a_1, a_2,...,a_n\in B(H)_{s.a.}$ with $\|a_i\|\le 1$ ($1\le i\le n$). In order to have $\|\Theta_{\lambda,\eta}\|\ge 1-\eta,$
it suffices to have one unit vector $x\in H$  such that
\beq\label{Test}
\la \Theta_{\lambda, \eta}(a_1,a_2,...,a_n)x, x\ra>1-\eta.
\eneq

\begin{df}\label{Dcpc} 
Let $A$ and $B$ be \CA s.  A linear map $L: A\to B$
is a c.p.c. map if it is completely positive and contractive.

\end{df}

\begin{df}\label{Dappsp}
Let $A$ be a  unital  \CA\, and $a_1,a_2,...,a_n\in A_{s.a.}$ for some $n\in \N.$
Let us assume that $\|a_i\|\le M$ (for some $M>0$), $i=1,2,...,n.$ 
Fix $0<\eta<1/2.$ Suppose  that $X\subset M^n$ is a compact subset 
and $L: C(X)\to A$ is a unital c.p.c. map
such that 
\beq
%
\hspace{-2.1in}{\rm (i)} &&\|L(e_j|_X)-a_j\|<\eta,\,\,\, 1\le j\le n, \\
\hspace{-2.1in}{\rm (ii)}&&\|L((e_je_i)|_X)-L(e_i|_X)L(e_j|_X)\|<\eta, \,\,\, 1\le i,j\le n,\andeqn\\
\hspace{-2.1in}{\rm (iii)} &&  \|L(f)\|\ge 1-\eta
\eneq
for any $f\in C(X)_+$  which has 
value $1$ on an open ball with the center $x\in X$ (for some $x$) and the radius $\eta.$
Then we say that $X$ is an $\eta$-near-spectrum of  the $n$-tuple 
$(a_1, a_2,...,a_n).$ 
We write 
$$
nSp^\eta((a_1,a_2,...,a_n)):=X.
$$

If, moreover,  $L$ is a unital \hm,  then  we say $X$ is 
an $\eta$- spectrum of the $n$-tuple $(a_1,a_2,...,a_n).$

Let $X$ be an $\eta$-near-spectrum and 
$Y$ be  a $\dt$-near -spectrum  for $(a_1,a_2,...,a_n),$ respectively.
Suppose that $\dt<\eta,$ then, by the definition, 
$Y$ is also an $\eta$-near-spectrum. 
In particular, if $nSp^\dt((a_1,a_2,...,a_n))\not=\emptyset,$ then 
$nSp^\eta((a_1,a_2,...,a_n))\not=\emptyset.$

\end{df}

\begin{prop}[Proposition 2.11 of \cite{Linself}] \label{Pappsp}
Fix $n\in \N.$ For any $\eta>0,$ there exists $\dt(n,\eta)>0$ satisfying the following:
Suppose that $A$ is a  unital \CA\, and $a_i\in A_{s.a.}$ with $\|a_i\|\le 1,$ $1\le i\le n,$
such that
\beq
\|a_ia_j-a_ja_i\|<\dt,\,\,\, 1\le i,j\le n.
\eneq
Then 

(1) 
$X:=s{\rm Sp}^\eta((a_1,a_2,...,a_n))\not=\emptyset$ and 

(2) $Y:=nSp^\eta((a_1,a_2,...,a_n))\not=\emptyset.$
%
%
\end{prop}

\begin{prop}[Proposition 2.15  of \cite{Linself}] \label{Pspuniq}
Fix $k\in \N.$  For any $\eta>0,$ there exits $\dt(k, \eta)>0$ satisfying the following:

Suppose that $A$ is a unital \CA\, and $a_1, a_2,...,a_k\in A_{s.a.}$ with $\|a_i\|\le 1$ 
($1\le i\le k$)
such that $(a_1, a_2,...,a_k)$ has 
a non-empty $\dt$-near-spectrum $X=nSp^\dt((a_1,a_2,...,a_n)).$ 
If 
$Y$ is also a non-empty  $\dt$-near-spectrum of $(a_1,a_2,...,a_k),$
and $Z=s{\rm Sp}^\eta((a_1,a_2,...,a_k))\not=\emptyset,$ 
then 
\beq
d_H(X, Y)<\eta\andeqn  X, Y\subset Z\subset X_{2\eta}.
\eneq
\end{prop}

\section{The case of small commutators}

\begin{lem}\label{LXX}
Let $\Omega$ be a compact metric space, $X\subset \Omega$ a compact subset and 
$X_i\subset \Omega$ be a sequence of compact subsets.
Suppose that 
\beq\label{LXX-0}
\lim_{i\to\infty}{\rm dist}_H(X_i, X)=0.
\eneq
Suppose that ${\cal G}\subset C(\Omega)$ is a finite subset 
which generates $C(\Omega)$ and 
there are a sequence of decreasing numbers $\eta_i\searrow 0,$ 
and  a sequence of c.p.c. maps $L_i: C(\Omega)\to A_i$ for some unital \CA\, $A_i$ such that
\beq\label{LXX-1}
&&\| L_i(fg)-L_i(f)L_i(g)\|<\eta_i\tforal f,g\in {\cal G}\tand\\\label{LXX-2}
&&\|L_i(h)\|\ge 1-\eta_i
\eneq
for any $h\in C(X)_+$ with value 1 in an open ball with center at some point $x\in  X_i$ and 
radius $\eta_i,$ $i=1,2,....$

Let $\Psi: C(\Omega)\to \prod_{i=1}^\infty A_i/\bigoplus_{i=1}^\infty A_i$
be the \hm\, defined by $\Psi(f)=\Pi\circ \{L_i(f)\}$ for all $f\in C(\Omega),$
where $\Pi:  \prod_{i=1}^\infty A_i\to  \prod_{i=1}^\infty A_i/\bigoplus_{i=1}^\infty A_i$ is the quotient map.
Then  ${\rm ker}\Psi=\{f\in C(\Omega): f|_X=0\}.$
\end{lem}

\begin{proof}
First we note, by \eqref{LXX-1},  that $\Psi=\Pi\circ \{\L_i\}$ is indeed a \hm, since 
${\cal G}$ is a generating set.

Let $f\in C(\Omega)$ such that $f|_X=0.$
To show that $\Psi(f)=0,$ let $\ep>0.$ 
Since $X$ is compact, there is $\dt>0$ such that
\beq
\|f(x)\|<\ep \rforal x\in \{x\in \Omega: {\rm dist}(x, X)<\dt\}.
\eneq
By \eqref{LXX-0}, there exists $i_0\in \N$  such that
\beq
X_i\subset \{x\in \Omega: {\rm dist}(x, X)<\dt\}\rforal i\ge i_0.
\eneq
Hence 
\beq
\|L_i(f|_{X_i})\|<\ep\rforal i\ge i_0.
\eneq
It follows that 
\beq
\|\Psi(f)\|<\ep.
\eneq
Since this holds for any $\ep>0,$ we have $\Psi(f)=0.$
Hence 
\beq
\{f\in C(\Omega): f|_X=0\}\subset {\rm ker}\Psi.
\eneq

To complete the proof of the lemma, it suffices to show that 
if $f\not\in \{g\in C(\Omega): g|_X=0\},$ then $\Psi(f)\not=0.$
Since $I=\{g\in C(\Omega): g|_X=0\}$ is an ideal and $\Psi(f)^*\Psi(f)=\Psi(f^*f),$
it suffices 
show that
if $f\not\in \{g\in C(\Omega): g|_X=0\}$ and $f\ge 0,$ then $\Psi(f)\not=0.$

Thus we assume that $f\ge 0.$ 
Then there is $x\in X$ such that
$f(x)>0.$  

By considering $(1/f(x))f,$ to simplify the notation, we may assume 
that $f(x)=1.$  There exists $\dt_1>0$ such that
$|f(\zeta)|\ge 1/2$ for all $\zeta\in B(x, \dt_1).$ Choose $h\in C(\Omega)_+$
such that  $0\le h\le 1,$ $h(y)=1$ if $y\in B(x, \dt_1/2)$ and $h(y)=0$ if 
$y\not\in B(x, \dt_1).$ 
Hence $ 2f\ge h.$ 

By \eqref{LXX-0}, there exists $i_1\in \N$ such that, for all $i\ge i_0,$ there exists $x_i\in X_i$
such that
\beq
{\rm dist}(x_i, x)<\dt_1/8 \andeqn \eta_i<\dt_1/16.
\eneq
Hence $B(x_i,\dt_1/8)\subset B(x, \dt_1/2).$ 
It follows from \eqref{LXX-2} that
\beq
\|L_i(2f)\|\ge \|L_i(h)\|\ge 1-\eta_i\rforal i\ge i_1.
\eneq
Hence $\Psi(2f)\ge 1$ and $\Psi(f)\not=0.$
\end{proof}

\begin{lem}\label{Lcorona}
Let  $X$ be a compact metric space and $B$ be a \CA\, which is a corona algebra of real rank zero.
Suppose that $\phi: C(X)\to B$ is a unital injective \hm. 
Then, for any finite subset $S=\{\zeta_1, \zeta_2,...,\zeta_m\}\subset X,$ 
there are mutually orthogonal non-zero projections 
$d_1, d_2,...,d_m\in B$  with $d=\sum_{i=1}^m d_i$ such that
\beq
&&(1-d)\phi(f)=\phi(f)(1-d)\tand\\
&&\phi(f)=\sum_{i=1}^m f(\zeta_i)d_i+(1-d)\phi(f)(1-d)\rforal f\in C(X).
\eneq
\end{lem}

\begin{proof}
For each $\zeta_k\in X,$
let 
\beq
C_k=\{f\in C(X): f(\zeta_k)=0\},
\eneq
an ideal of $C(X).$
Put $D_k=\overline{\phi(C_k)B\phi(C_k)}.$  Then $D_k$ is a $\sigma$-unital hereditary \SCA\, of $B.$
Since $\phi$ is injective,  no elements in $D_k$ is invertible.
Hence $D_k\not=B.$
Since $B$ is a corona algebra, by Pedersen's double annihilator theorem (Theorem 15 of \cite{PedSAW}),
\beq
D_k=(D_k^\perp)^\perp.
\eneq
 It follows that $D_k^\perp=\{c\in B: cd=dc=0\rforal d\in D_k\}\not=\{0\}.$
 Since $B$ has real rank zero, the
hereditary \SCA\, $D_k^\perp$ also has real rank zero (Corollary  2.8 of \cite{BP}).
 Choose a non-zero projection $d_k\in D_k^\perp,$ $k=1, 2,...,m.$
 Note that $d_jd_k=0=d_kd_j$ if $k\not=j.$

Let $F_k=\{x\in B: xd,dx\in D_k\}$ be the idealizer of $D_k$ in $B$ and 
 $E_k=\phi(C(X))+D_k.$  Then $E_k\subset F_k$ and $d_k\in F_k.$
 Let $\pi_k: F_k\to F_k/D_k$ be the quotient map. 
 Then $\pi_k\circ \phi(f)=f(\zeta_k)$ for all $f\in C(X).$ 
 It follows that $d_k\phi(f)=\phi(f)d_k=f(\zeta_k)d_k$ for all $f\in C(X),$  $k=1,2,...,m.$
 Define $d=\sum_{k=1}^m d_k.$
 Then 
 \beq
 &&d\phi(f)=\phi(f)d=\sum_{k=1}^m f(\zeta_k)d_k\andeqn\\
 &&\phi(f)=\sum_{k=1}^m f(\zeta_k)d_k+(1-d)\phi(f)(1-d)\rforal f\in C(X).
 \eneq
\end{proof}

\begin{lem}\label{Ldig}
Let $\ep>0,$ $\Omega$ be a compact metric space, and 
${\cal G}\subset C(\Omega)$ be a finite generating set. 
 There exists $\dt(n, \ep)>0$ satisfying the following:
Suppose that $A$ is a unital \CA\, of real rank zero and $X\subset \Omega$
is a non-empty compact subset 
 and 
$\phi: C(X)\to A$ is a c.p.c. map such that
\beq\nonumber
\|\phi(g_1|_X g_2|_X)-\phi(g_1|_X)\phi(g_2|_X)\|<\dt \tforal g_1,g_2\in {\cal G},\,\,\, 1\le i,j\le n\tand
\|\phi(f)\|\ge 1-\dt
\eneq
for any $f\in C(X)_+$  which has 
value $1$ on an open ball with the center $x\in X$ (for some $x$) and the radius $\dt.$
Then, 
there are $\xi_1, \xi_2,..., \xi_m\in X$ which is $\ep$-dense in 
$X,$
 and 
mutually orthogonal non-zero projections 
$p_1,p_2,..., p_m\in A$ such that, for all $f\in {\cal G},$
\beq
\|\sum_{k=1}^m f(\xi_k) p_k+(1-p)\phi(f)(1-p)-\phi(f)\|<\ep.
\eneq 
\end{lem}

\begin{proof}
We prove this by contradiction. So we  assume the lemma is false.
Then, we obtain positive numbers $\ep_0>0,$ 
and $\ep_1>0,$  
a sequence of unital \CA s $\{A_i\}$ of real rank zero, 
a sequence of compact subsets $X_i\subset \Omega,$ 
a sequence of c.p.c. maps $L_i: C(X_i)\to A_i$ ($i\in \N$)  and 
a sequence of positive numbers $\eta_i$ such that $\eta_i\searrow 0$ 
such that
\beq\label{LLd-2}
&&\lim_{i\to\infty}\|L_i(fg|_{X_i})-L_i(f|_{X_i})L_i(g|_{X_i})\|=0\rforal f, g\in {\cal G}\andeqn\\\label{LLd-3}
&&\|L_i(h)\|\ge 1-\eta_i
\eneq
for any $h\in C(X_i)_+$ which has value 1 on an open ball of radius $\eta_i$  with center at a 
point $x\in X_i,$ but 
\beq\label{LLd-5}
\inf\{\sup\{\|\sum_{k=1}^{m_i} g(\xi_{i,k}) p_{k,i}+(1-q_i)L_i(g|_{X_i})(1-q_i)-L_i(g|_{X_i})\|: g\in {\cal G}\}\}\ge \ep_0,
\eneq 
where the  infimum is taken among all mutually orthogonal non-zero projections 
$\{p_{i,1}, p_{i,2},...,p_{i, m_i}\}$ in $A_i$ with
$q_i=1-\sum_{k=1}^{m_i}p_{i,k}$  
and all finite $\ep_1$-dense subset $\{\xi_{i,1},\xi_{i,2},...,\xi_{i,m_i}\}$
of $X_i.$

Note that $(F(\Omega), d_H)$ is compact.
Therefore, by passing to a subsequence, we may assume 
that there is a compact subset $X\subset \Omega$ such 
that
\beq\label{LLd-6+1}
\lim_{i\to\infty}{\rm dist}_H(X_i, X)=0.
\eneq

Let $B=\prod_{i=1}^\infty A_i,$  
$\Phi: C(\Omega)\to B$ be defined by $\Phi(f)=\{L_i(f|_{X_i})\}$ for $f\in C(\Omega),$ and 
$\Pi: B\to B/\bigoplus_{i=1}^\infty A_i$ be the quotient map (note that
$\bigoplus_{i=1}^\infty A_i=\{\{a_i\}: a_i\in A_i\andeqn \lim_i\|a_i\|=0\}$).
Then, by \eqref{LLd-2},  $\Psi'=\Pi\circ  \Phi': C(\Omega)\to B/\bigoplus_{i=1}^\infty A_i$ is a \hm. 
Define $L_i': C(\Omega)\to A_i$ by $L_i'(f)=L_i(f|_{X_i})$ for $f\in C(\Omega).$ 
By \eqref{LLd-6+1}, \eqref{LLd-3} and by Lemma \ref{LXX}, 
\beq
{\rm ker}\Psi'=\{f\in C(\Omega): f|_X=0\}.
\eneq
Thus  $\Psi'$ induces an injective \hm\, 
$\Psi: C(X)\to B/\bigoplus_{i=1}^\infty A_i$  (with $\Psi(f|_X)=\Psi'(f)$ for all 
$f\in C(\Omega)$).

Let $\{\zeta_1, \zeta_2,...,\zeta_m\}\subset X$ be 
an $\ep_1$-dense
subset.
Put $Q= B/\bigoplus_{i=1}^\infty A_i.$
Note that $B=M(\bigoplus_{i=1}^\infty A_i),$ the multiplier algebra of $\bigoplus_{i=1}^\infty A_i.$
So $Q=M(\bigoplus_{i=1}^\infty A_i)/\bigoplus_{i=1}^\infty A_i$ is a corona algebra. 
 By applying Lemma \ref{Lcorona}, we obtain mutually orthogonal non-zero projections 
 $d_1, d_2,...,d_m$ with $d=\sum_{k=1}^m d_k$ such that 
 \beq
 &&d\Psi(f)=\Psi(f)d=\sum_{k=1}^m f(\zeta_k)d_k\andeqn\\
 &&\Psi(f)=\sum_{k=1}^m f(\zeta_k)d_k
 +(1-d)\Psi(f)(1-d)\rforal f\in C(X).
 \eneq
 Note that every projection in $\prod_{i=1}^\infty A_i/\oplus_{i=1}^\infty A_i$ 
 lifts to a projection in $\prod_{i=1}^\infty A_i.$ 
 So there is a projection $p_1\in B$ such that $\Pi(p_1)=e_1.$
 An induction let us find mutually 
 orthogonal non-zero projections $p_1, p_2,...,p_m$ in $B$  (and $p=\sum_{k=1}^m p_i$) such that
 $\Pi(p_k)=d_k$ ($1\le k\le m$) and
 \beq
 &&\hspace{-0.6in}\Pi(\sum_{k=1}^m f(\zeta_k)p_k+(1-p)\Phi(f)(1-p))\\
 &&\hspace{0.4in}=\sum_{k=1}^m f(\zeta_k)d_k
 +(1-d)\Psi(f)(1-d)=\Psi(f)\hspace{0.2in}\rforal f\in C(X).
 \eneq
Write $p_k=\{p_{i,k}\}_{i\in \N}$ and $p=\{q_{i}\}_{i\in \N}$ in $B.$ 
Then, for all sufficiently large $i,$
\beq
\|\sum_{k=1}^m g(\zeta_k)p_{i,k}+(1-q_i)L_i(g|_{X_i})(1-q_i)-L_i(g|_{X_i})\|<\ep_0/2\rforal g\in {\cal G}.
\eneq
 This contradict \eqref{LLd-5}. The lemma then follows.
\end{proof}

\begin{thm}\label{T1}
Let $n\in \N$ and $\ep>0.$  There exists $\dt(n, \ep)>0$ satisfying the following:
Suppose that $H$ is an infinite dimensional separable Hilbert space and $T_1, T_2,...,T_n\in B(H)$
are self-adjoint operators with $\|T_j\|\le 1$ ($1\le j\le n$) such that
\beq\label{T1-c}
\|T_iT_j-T_jT_i\|<\dt,\,\,\,\, 1\le i,j\le n.
\eneq
Then, for any $\lambda=(\lambda_1, \lambda_2,...,\lambda_n)\in s{\rm Sp}^{\ep/4}((T_1,T_2,...,T_n)),$
there exists $v\in H$ with $\|v\|=1$ such that
\beq
\max_{1\le i\le n}|{\rm exp}_{T_i}(v)-\lambda_i|<\ep\andeqn
\max_{1\le i\le n}\|(T_i-{\rm exp}_{T_i}(v)\cdot I)v\|<\ep.
\eneq
Moreover, for any $\mu=(\mu_1, \mu_2,...,\mu_n)\in {\rm conv}(s{\rm Sp}^{\ep/4}((T_1,T_2,...,T_n))),$ 
there is a vector $x\in H$ with $\|x\|=1$ such that
\beq\label{T1-c2}
(\sum_{i=1}^n|\mu_i-{\rm exp}_{T_i}(x)|^2)^{1/2}<\ep,\,\,\, 1\le i\le n.
\eneq
\end{thm}

\begin{proof}
We choose the generating set ${\cal G}=\{e_j: 0\le j\le n\}$ of $C(\I^n)$ (see   Definition \ref{DSP}).
Fix $0<\ep<1.$   Choose $\eta_1:=\dt(\ep/16, n)>0$ given by Lemma \ref{Ldig} 
for $\ep/16$ (in place of $\ep$), $\Omega=\I^n$ and ${\cal G}$ chosen above.
Put $\eta=\min\{\ep/4, \eta_1\}.$
Let $\dt_2:=\dt(\eta/4, n)>0$ be given by Proposition \ref{Pappsp} (see Proposition 
 2.11 of \cite{Linself})
for $\eta/4$ (in place of $\eta$).  Let $\dt_3:=\dt(n, \dt_2)>0$ 
be given by Proposition \ref{Pspuniq} 
(see Proposition 2.15 of \cite{Linself}) for $\eta/4$ (in place of $\eta$).
Put $\dt=\min\{\ep/4, \eta/4, \dt_2/2, \dt_3\}>0.$

Now assume that \eqref{T1-c} holds for the above $\dt.$ 
By Proposition \ref{Pappsp} (Proposition 2.11 of \cite{Linself}) and by the choice of $\dt$ (and  $\dt_2$),
\beq\label{MT-4}
X:=n{\rm Sp}^{\eta/4}((T_1,T_2,...,T_n))\not=\emptyset\andeqn 
Z:=s{\rm Sp}^{\eta/4}((T_1, T_2,...,T_n))\not=\emptyset.
\eneq
Applying Proposition \ref{Pspuniq},
we also have 
\beq\label{MT-5}
X\subset Z\subset X_{\eta/2}\subset \I^n.
\eneq
By the definition of $\eta/4$-near spectrum, we obtain a unital c.p.c. map
$\phi: C(X)\to A$ such that
\beq
&&\|\phi(e_j|_X)-T_j\|<\eta/4,\,\,\, 1\le j\le n，\\
&&\|\phi(e_ie_j|_X)-\phi(e_i|_X)\phi(e_j|_X)\|<\eta/4,\,\,\, 1\le i,j\le n\andeqn\\
&&\|\phi(f)\|\ge 1-\eta/4
\eneq
for all $f\in C(X)_+$ which has value 1 on an open ball of radius $\eta/4$ with center $x\in X.$
By the choice of $\eta,$ applying  Lemma \ref{Ldig},
we obtain an $\ep/4$-dense subset $\{\xi_1, \xi_2,...,\xi_m\}\subset X$ and 
mutually orthogonal non-zero projections $p_1, p_2,...,p_n\in B(H)$ such that
\beq\label{MT-10}
\|\sum_{k=1}^m e_j(\xi_k) p_k+(1-p)\phi(e_j|_X)(1-p)-T_j\|<\ep/4,\,\,\, 1\le j\le n.
\eneq 

Next let $\lambda=(\lambda_1, \lambda_2,...,\lambda_n)\in s{\rm Sp}^{\ep/4}((T_1, T_2,...,T_n))=Z.$ 
Then, by \eqref{MT-5},  there is $k\in \{1,2,...,m\}$
such that
\beq
\sqrt{\sum_{i=1}^n|\lambda_i-\xi_{k,i}|}<\eta/4<\ep/4,
\eneq
where $(\xi_{k,1}, \xi_{k,2},...,\xi_{k,n})=\xi_k.$

Choose $v_k\in p_k(H)$ with $\|v_k\|=1.$ 
Then, by \eqref{MT-10},
\beq
|{\rm exp}_{T_j}(v_k)-\xi_{k,j}|&=&|\la T_jv_k,v_k\ra -\xi_{k,j}|<\ep/4+|\sum_{i=1}^m \la e_j(\xi_i)p_iv_k,v_k\ra -\xi_{k,j}|\\\label{MT-11}
&=&\ep/4+|\xi_{k,j}\la p_kv_k,v_k\ra-\xi_{k,j}|=\ep/4.
\eneq
Hence
\beq
|{\rm ep}_{T_j}(v_k)-\lambda_{k,j}|<\ep/4+|\xi_{k,j}-\lambda_{k,j}|<\ep/4+\ep/4=\ep/2<\ep.
\eneq
Therefore, for each  $1\le j\le n,$ by \eqref{MT-10}, 
\beq\nonumber
\|(T_j-\xi_{k,j}\cdot I)v_k\| &<&
\ep/4+\|(\sum_{i=1}^m e_j(\xi_i) p_i+(1-p)\phi(e_j|_X)(1-p))v_k-\xi_{k,j}v_k\|\\
&=&\ep/4+\|(\sum_{i=1}^m e_j(\xi_i)p_i)v_k-\xi_{k,j}v_k\|\\\label{MT-16}
&=&\ep/4+\|\xi_{k,j}v_k-\xi_{k,j}v_k\|=\ep/4.
\eneq
Hence,  by \eqref{MT-11},
\beq\nonumber
\|(T_j-{\rm exp}_{T_j}(v_k)\cdot I)v_k\| &<&
\ep/4+\|(T_j-\xi_{k,j}\cdot I)v_k\|<\ep/4+\ep/4=\ep/2.
\eneq
Next, let $\mu=\sum_{l=1}^Na_l \zeta_l,$ where $a_l\ge 0$ and 
$\sum_{l=1}^Na_l=1$ and $\zeta_l\in s{\rm Sp}^{\ep/4}((T_1, T_2,...,T_n)).$
Since $\{\xi_1, \xi_2,...,\xi_m\}$ is $\ep/4$-dense, 
we may assume that there are $\af_i\ge 0$ with 
$\sum_{i=1}^m \af_i=1$ such that
\beq\label{MT-20}
\|\mu-\sum_{i=1}^m \af_i\xi_k\|_2<\ep/4,
\eneq
where $\|\cdot\|_2$ is the Euclidian norm in $\R^n.$
Choose 
$
x=\sum_{k=1}^m \sqrt{\af_k}v_k\in H.
$
Since $v_i\perp v_j,$ when $i\not=j,$ we have 
$\|x\|^2=\la x, x\ra=\sum_{k=1}^m\af_k\|v_k\|^2=1.$
We then compute that, by \eqref{MT-16} and \eqref{MT-20},
\beq
{\rm exp}_{T_j}(x)=\la T_j x, x\ra=\sum_{k=1}^m \sqrt{\af_k}\la T_jv_k, x\ra
\approx_{\ep/4}  \sum_{k=1}^m\sqrt{\af_k}\la \xi_{k,j}v_k,x\ra\\
=\sum_{k=1}^m \af_k\xi_{k,i}\la v_k, v_k\ra=\sum_{k=1}^m \af_k\xi_{k,i}\approx_{\ep/4}\mu_i.
\eneq
Thus \eqref{T1-c2} holds.
\end{proof}

Theorem \ref{IT-1} follows immediately from Theorem \ref{T1}.

\section{The case of compact commutators}

\begin{thm}\label{TC+ep}
Let $\ep>0$ and $n\in \N.$ 
There is $\dt(\ep,n)>0$ satisfying the following:
Suppose that $H$ is an infinite dimensional separable Hilbert space.
Suppose that $T_1, T_2,...,T_n$ are in $B(H)_{s.a.}$ with $\|T_i\|\le 1$ ($1\le i\le n$) such that
\beq\label{TC+ep-0}
\|\pi_c(T_iT_j)-\pi_c(T_iT_j)\|<\dt,\,\,\, 1\le i,j\le n.
\eneq
Then, for any $\lambda=(\lambda_1, \lambda_2,...,\lambda_n)\in 
s{\rm Sp}^{\ep/4}((\pi_c(T_1),\pi_c(T_2),...,\pi_c(T_n))),$
there exists $v\in H$ with $\|v\|=1$ such that
\beq\label{LT-01}
\max_{1\le i\le n}|{\rm exp}_{T_i}(v)-\lambda_i|<\ep\tand
\max_{1\le i\le n}\|(T_i-{\rm exp}_{T_i}(v)\cdot I)v\|<\ep.
\eneq
Moreover, for any $\mu=(\mu_1, \mu_2,...,\mu_n)\in {\rm conv}(s{\rm Sp}^{\ep/4}((\pi_c(T_1),\pi_c(T_2),...,\pi_c(T_n))),$
there exists a vector $x\in H$ with $\|x\|=1$ such that
\beq
(\sum_{i=1}^n |\mu_i-{\rm exp}_{T_j}(x)|^2)^{1/2}<\ep
\eneq
\end{thm}

\begin{proof}
Let $\dt_0>0$ (in place of $\dt$) be given by Proposition \ref{Pappsp} 
(see Proposition 2.11 of \cite{Linself}) for $\ep/4$ (in place of $\eta$).
Let $\dt_1>0$ (in place of $\dt$) be given by Proposition \ref{Pspuniq} (Proposition 2.15 of \cite{Linself})
for $\ep/16$ (in place of $\eta$).
Put $\eta=\min\{\dt_1/2, \dt_0/2, \ep/32\}>0.$ 
Choose $\dt_2>0$ (in place of $\dt$) be given by Corollary 1.20 of \cite{Linalm97}  for $\eta$ (in place of $\ep$).
Put $\dt=\min\{\dt_0, \dt_1, \dt_2\}.$

Suppose that \eqref{TC+ep-0} holds for this $\dt.$
 Recall that $B(H)/{\cal K}$
is a purely infinite simple \CA. Then, by applying Corollary 1.20  of \cite{Linalm97},
we obtain $s_1, s_2,...,s_n\in (B(H)/{\cal K})_{s.a.}$ such that
\beq\label{Tcp+ep+5}
s_is_j=s_js_i\andeqn \|s_i-\pi_c(T_i)\|<\eta,\,\,\, 1\le i, j\le n.
\eneq
By the choice of $\eta,$  applying  Proposition \ref{Pappsp},  we conclude that
$Z:=s{\rm Sp}^{\ep/4}((\pi_c(T_1),\pi_c(T_2),...,\pi_c(T_n)))\not=\emptyset.$ 
Let $C\subset B(H)/{\cal K}$ be the unital \SCA\, generated by $s_1,s_2,...,s_n,1.$
Then there exists a unital injective \hm\, $\psi: C(X)\to B(H)/{\cal K}$
for some compact subset $X\subset \I^n$ such that
$\psi(e_i|_X)=s_i,$ $1\le i\le n.$
Then $X$ is an $\eta$-spectrum for $(\pi_c(T_1), \pi_c(T_2),...,\pi_c(T_n)).$ 

By Proposition \ref{Pspuniq} (with our choice of $\dt_1$),
\beq
{\rm dist}_H(X, Z)<2(\ep/16)=\ep/8.
\eneq
Choose an $\ep/16$-dense subset $\{\zeta_k=(\zeta_{k,1},\zeta_{k,2},...,\zeta_{k,n}): 1\le k\le m\}$
of $X.$ 
By applying Lemma \ref{Lcorona},
there are mutually orthogonal non-zero projections 
$d_1, d_2,...,d_m$ with $d=\sum_{j=1}^md_j$ such that
\beq
(1-p)\psi(f)=\psi(f)(1-p)\andeqn 
\psi(f)=\sum_{j=1}^m f(\zeta_j)d_k+(1-p)\psi(f)(1-p)
\eneq
for all $f\in C(X).$ By \eqref{Tcp+ep+5}, there are $K_i\in {\cal K}$ ($1\le i\le n$) 
and mutually orthogonal projections $p_1, p_2,...,p_m$
such that $\pi_c(p_j)=d_j$ ($1\le j\le m$) and 
\beq\label{475}
\| \sum_{j=1}^m e_i(\zeta_j)p_j+(1-p)T_i(1-p)+K_i-T_i\|<\eta,\,\,\, 1\le i\le n.
\eneq
 Let $\{E_{k,0}\}$ be an approximate identity of $(1-p){\cal K}(1-p)$
and $\{E_{k,j}\}$ be an approximate identity of $p_j{\cal K}p_j$ ($1\le j\le m$) consisting of projections
($0\le j\le m$), respectively.
Put $E_k=E_{k,0}+\sum_{j=1}^m E_{k,j},$ $k\in \N.$
Then $\{E_k\}$ forms an approximate identity for ${\cal K}.$
We may choose $E_k$ such that
\beq\label{T2-12}
\|(1-E_k)K_i\|<\ep/8\andeqn \|K_i(1-E_k)\|<\ep/8,\,\,\, i=1,2,...,n.
\eneq
Note that 
\beq\label{T2-13}
p_j-E_{k,j}=p_j(1-E_k)\le 1-E_k,\,\,\, 1\le j\le m.
\eneq
So we also have (by \eqref{475} and \eqref{T2-12}), for $1\le i\le n,$ 
\beq\label{T2-12+}
&&\hspace{-0.8in}\|\sum_{j=1}^m e_i(\zeta_j)(p_j-E_{k,j})+(1-p)T_i(1-p-E_{k,0})-T_i(1-E_k)\|<\eta+{\ep\over{8}}\andeqn\\
&&\hspace{-0.8in}\|\sum_{j=1}^m e_i(\zeta_j)(p_j-E_{k,j})+(1-p-E_{k,0})T_i(1-p-E_{k,0})-(1-E_k)T_i(1-E_k)\|<\eta+{\ep\over{8}}.
\eneq
Choose $v_j\in (p_j-E_{k,j})H$ with $\|v_j\|=1,$ $j=1,2,...,m.$
Then $v_k\perp v_j$ if $k\not=j.$ 
Hence, by \eqref{T2-13} and by \eqref{T2-12+} (and since $\eta+\ep/8<5\ep/32$)
\beq\label{T2-18}
T_iv_j\approx_{{5\ep\over{32}}} \sum_{j=1}^m e_i(\zeta_j)(p_j-E_{k,j})v_j=\zeta_{j,i}v_j.
\eneq

Next  let $\lambda=(\lambda_1,\lambda_2,...,\lambda_n)\in Z$ be given.
We may assume that, for some $\zeta_j,$  
\beq
(\sum_{i=1}^m |\zeta_{j,i}-\lambda_i|^2)^{1/2}<\ep/8+\ep/16=3\ep/16.
\eneq
By  \eqref{T2-18}, 
we estimate that
\beq\label{T2-19}
(\sum_{i=1}^n|{\rm exp}_{T_i}(v_j)-\lambda_i|^2)^{1/2}<
3\ep/16+
(\sum_{i=1}^n|{\rm exp}_{T_i}(v_j)-\zeta_{j,i}|^2)^{1/2}
<3\ep/16+5\ep/32.
\eneq
Moreover, by \eqref{T2-18}, for $1\le i\le n,$
\beq
\|(T_i-{\rm exp}_{T_i}(v_j)\cdot I)v_j\|<5\ep/32+\|\zeta_{j,i}v_j-{\rm exp}_{T_i}(v_j)v_j\|<\ep/2.
\eneq

Now let $\mu=\sum_{l=1}^Na_l \zeta_l,$ where $a_l\ge 0$ and 
$\sum_{l=1}^Na_l=1$ and $\zeta_l\in s{\rm Sp}^{\ep/4}((T_1, T_2,...,T_n)),$ $1\le l\le n.$
Since $\{\zeta_1, \zeta_2,...,\zeta_m\}$ is $\ep/16$-dense, 
we may assume that there are $\af_i\ge 0$ with 
$\sum_{i=1}^m \af_i=1$ such that
\beq
\|\mu-\sum_{i=1}^m \af_i\zeta_k\|_2<\ep/16,
\eneq
where $\|\cdot\|_2$ is the Euclidian norm in $\R^n.$
Choose 
$
x=\sum_{k=1}^m \sqrt{\af_k}v_k\in H.
$
Since $v_i\perp v_j,$ when $i\not=j,$ we have 
$\|x\|^2=\la x, x\ra=\sum_{k=1}^m\af_k\|v_k\|^2=1.$
We then compute that, by \eqref{T2-18} and 
\beq
{\rm exp}_{T_j}(x)=\la T_j x, x\ra=\sum_{k=1}^m \sqrt{\af_k}\la T_j v_k, x\ra \approx_{5\ep/8}  \sum_{k=1}^m\sqrt{\af_k}\la \zeta_{k,j} v_k,x\ra\\
=\sum_{k=1}^m \af_k\zeta_{k,i}\la v_k, v_k\ra=\sum_{k=1}^m \af_k\zeta_{k,i}\approx_{\ep/16}\mu_i.
\eneq
\end{proof}

{\bf The proof of Theorem \ref{IT2}}:
The condition that $T_iT_j-T_jT_i\in {\cal K}$ in \eqref{T1-c} means
$\|\pi_c(T_i)\pi_c(T_j)-\pi_c(T_j)\pi_c(T_i)\|=\dt$ for any $\dt>0.$  
Moreover ${\rm Sp}((T_1,T_2,...,T_n))\subset s{\rm Sp}^{\ep/4}((T_1, T_2,...,T_n))$
for any $\ep>0.$ Hence, by Theorem \ref{TC+ep}, 
\eqref{LT-01} holds for any $\ep>0.$
Hence  \eqref{IT2-1}, \eqref{IT2-2} and \eqref{IT2-3} all hold.

\vspace{0.2in}

\begin{rem}\label{catdiscussion}
{\rm  Let $H$ be an infinite dimensional separable Hilbert space and 
$T_1, T_2,...,T_n\in B(H)$ be self-adjoint operators.

(1) Even in a classical system, a system in which the observables (self-adjoint 
operators $T_1, T_2,...,T_n$)  commute,
not every  $n$-tuple $\lambda=(\lambda_1, \lambda_2,...,\lambda_n),$ where 
each $\lambda_i$ is in the spectrum of $T_i,$ can be measured simultaneously.
The requirement is that $\lambda\in {\rm Sp}((T_1,T_2, ...,T_n))$ (see Definition \ref{Dsp1}). 
It is well known that spectrum of an operator is difficult to compute and unstable for 
small perturbations. 
When observables are not commuting,
it is  even difficult to define ``joint-spectrum". 
 However, for $\ep>0,$ $s{\rm Sp}^\ep((T_1, T_2,...,T_n))$ can be tested (see \eqref{Test}), and 
by Proposition \ref{Pappsp},  it is not empty, when commutators are small.
So the $\ep$-synthetic-spectrum $s{\rm Sp}((T_1, T_2,..,T_n))$ is a workable 
replacement.

(2) Let $s_1$ and $s_2$  be AMU states corresponding to joint 
expected values $\lambda=(\lambda_1, \lambda_2,...,\lambda_n)$ and 
$\mu=(\mu_1,\mu_2,...,\mu_n),$  respectively.
Then the superposition  $s={s_1+s_2\over{\|s_1+s_2\|}}$ gives 
${\rm exp}_{T_i}(s)={\lambda_i+\mu_i\over{\|s_1+s_2\|}},$ $i=1,2,...,n.$ 
In fact, for any finitely many AMU states (or any states) $s_k: 1\le k\le m $ 
with joint expected values $\zeta_k=(\zeta_{k,1}, \zeta_{k,2},...,\zeta_{k,n})\in  \R^n,$ 
any $\af_k\ge 0$ ($1\le k\le m$) with $\sum_{k=1}^m \af_k=1,$ 
 $s=\sum_{k=1}^m \af_i s_i$ is a state 
with joint expected value 
$\zeta=\sum_{k=1}^m \af_k\zeta_k.$  
This mixed state measures $\sum_{k=1}^m \af_k \zeta_{k,j},$ $1\le j\le n,$ 
simultaneously. 
 However, 
if $\zeta \in s{\rm Sp}^{\eta}((T_1, T_2,...,T_n))$ for some $\eta>0,$ it is not clear that 
$s$ should be called a  cat-state (as  $\zeta$ could be an eigenvalue).

(3) Let us assume that $\zeta\not\in s{\rm Sp}^{\eta}((T_1, T_2,...,T_n))$ for some small $\eta>0.$
This cat-state $s=\sum_{k=1}^m \af_i s_i$ is never  a pure state as pure states correspond to extremal 
states. 
The question is: 

\hspace{0.1in}

Could $\zeta$ be a joint expected value
measured by a vector state (simultaneously)?

\hspace{0.1in}

The ``Moreover" part of 
Theorem \ref{T1} (see also  that of Theorem \ref{TC+ep}) states that $\zeta$ could be approximately measured simultaneously 
by a vector state (but not by a AMU state).   This is an interesting feature since the $n$-tuple may  not
be approximated by  commuting ones.

(4) Under the assumption \eqref{IT2-0}, the commutators could have a large norm.
However, \eqref{compct-1} means that the commutators tend to zero along any orthonormal basis 
of $H.$ In fact the commutators vanish in the corona algebra $B(H)/{\cal K}$ (the standard quantum boundary).
 It is important to note that the assumption \eqref{IT2-0}
does not give us a compact perturbation of some $n$-tuple commuting self-adjoint operators.
More precisely, by the Brown-Douglas-Fillmore Theory,  there {\it cannot} be any 
self-adjoint operators $S_1,S_2,...,S_n$ with $S_jS_i=S_iS_j$ such that
\beq
T_j \,\,\,{\rm is \,\, \, close\,\, to}\,\,\,S_j+K_j,\,\,\, 1\le j\le n
\eneq
for some compact operators $K_j$ unless a Fredholm index vanishes (see, also Proposition 5.5 of \cite{Linself}) .
However, Theorem \ref{IT2} shows that there are plenty of AMU states.

Theorem \ref{TC+ep} goes somewhat beyond  both Theorem \ref{IT-1} and \ref{IT2} even though 
the the set of joint expected values given by AMU states might be smaller.

\begin{exm}\label{FRem}
Let $S$ be the shift operator on $H=l^2.$
Put $A_1=(S+S^*)/2$ and $A_2=-(S-S^*)/2i.$
Then $A_1$ and $A_2$ are selfadjoint operators and 
$\|A_1A_2-A_2A_1\|=1.$ But $A_1A_2-A_2A_1\in {\cal K}.$
However the pair $(A_1,A_2)$ is far away from commuting pairs and 
not even a compact perturbation of commuting pairs as $S$ is not a compact perturbation 
of any unitary. 
Nevertheless, by Theorem \ref{IT2},  for any $\lambda=(\lambda_1, \lambda_2)$ with $|\lambda|=1,$
there are unit vectors $v_k\in H$ ($k\in \N$) such that
\beq
\lim_{k\to \infty}\max_{1\le i\le 2}|\lambda_j-{\rm exp}_{A_j}(v_n)|=0\andeqn
\lim_{k\to\infty}\| (A_j-{\rm exp}_{A_j}(v_n)\cdot I)v_k\|=0.
\eneq
Moreover, for any $\zeta=(\zeta_1, \zeta_2)\in \R^2$ with $|\zeta|\le 1,$ there exists a sequence 
of unit vectors $u_m\in H$ such that
\beq
\lim_{k\to\infty}\max_{1\le i\le 2}|\zeta_j-\la A_j u_m, u_m\ra |=0.
\eneq
\end{exm}

 }
 
\end{rem}

\vspace{0.2in}

hlin@uoregon.edu

\end{document}